\documentclass[12pt,a4paper]{article}
\usepackage[utf8]{inputenc}
\usepackage[T2A]{fontenc}
\usepackage{amsmath, amssymb, amsthm}
\usepackage[width=6.50in, height=8.00in, left=1.00in, top=1.00in]{geometry}
\usepackage[english]{babel}

\newtheorem{thm}{Theorem}
\newtheorem{lemma}{Lemma}
\theoremstyle{remark}
\newtheorem{remark}{Remark}

\usepackage{dsfont}

\DeclareMathOperator*{\wlim}{\mathit{w}-lim}
\DeclareMathOperator{\ReC}{Re}

\title{Perturbation of an $\alpha$-stable type stochastic process by a pseudo-gradient}

\author{Mykola Boiko, Mykhailo Osypchuk}

\date{\footnotesize{Vasyl Stefanyk Precarpathian National University,\\ 57, Shevchenko str, 76018 Ivano-Frankivsk, Ukraine}}

\begin{document}
\maketitle
\begin{abstract}
	We consider a Markov process defined by some pseudo-differential operator of an order $1<\alpha<2$ as the process generator. Using a pseudo-gradient operator, that is, the operator defined by the symbol $i\lambda|\lambda|^{\beta-1}$ with some $0<\beta<1$, the perturbation of the Markov process under consideration by the pseudo-gradient with a multiplier, which is integrable at some great enough power, is constructed. Such perturbation defines a  family of evolution operators,  properties of which are investigated. A corresponding Cauchy problem is considered.
		\footnotetext{\textbf{Keywords:} $\alpha$-stable process, perturbation, pseudo-gradient, pseudo-process\par
		\textbf{MSC2010:} 60G52}
	\end{abstract}

\section{Introduction}\label{BO_sec1}
Let $d$ be some fixed positive integer number. By  $\mathbb{R}^d$ we denote the real $d$-dimensional Euclidean space.  As usual, we denote by $(\cdot,\cdot)$ the inner product and by $|\cdot|$ the norm in $\mathbb{R}^d$ (we use the last notation for denoting the absolute value of a real number and the module of a complex number).

Let us consider a family of pseudo-differential operators $(A(t,x))_{t\ge0,x\in\mathbb{R}^d}$ defined by the symbols $(a(t,x,\lambda))_{\lambda\in\mathbb{R}^d}$ for every $t\ge0$, $x\in\mathbb{R}^d$. That is,
\[
A(t,x)f(x)=\frac{1}{(2\pi)^d}\int_{\mathbb{R}^d}a(t,x,\lambda)F(\lambda)e^{i(\lambda,x)}\,\mathrm{d}\lambda, \quad t\le0,\ x\in\mathbb{R}^d,
\]
where $F$ is the Fourier transform of the function $f$: $F(\lambda)=\int_{\mathbb{R}^d}f(x)e^{-i(\lambda,x)}\,\mathrm{d}x$, $\lambda\in\mathbb{R}^d$.
Note that the function $a$ might be complex valued. 

We assume the conditions from \cite[Ch. 4]{Kochubey} about the function $a$ and the operator $A$, that is, (we denote by $Lf(r,\cdot)(x)$ the result of the operator $L$ used on $f(r,x)$ as the function of $x$) 
\begin{enumerate}\renewcommand{\theenumi}{$A_\arabic{enumi}$}\renewcommand{\labelenumi}{\theenumi)}
	\item\label{A1} The function $a$ is homogeneous of a degree $\alpha\in(1,2)$ with respect to the variable $\lambda$ and $\ReC a(t,x,\lambda)\ge a_0>0$ for all $t\ge0$, $x, \lambda\in\mathbb{R}^d$, $|\lambda|=1$.
	\item\label{A2} The function $a$ has $N\ge 2d+3$ (it is some natural number) continuous derivatives in $\lambda$ for all $\lambda\neq0$, and
	\[
	|\partial^\varkappa a(t,x,\cdot)(\lambda)|\le C_{N}|\lambda|^{\alpha-|\varkappa|},
	\]
	\[
	|\partial^\varkappa[a(t,x,\cdot)-a(s,y,\cdot)](\lambda)|\le C_{N}\left(|x-y|^\gamma+|t-s|^{\gamma/\alpha}\right)|\lambda|^{\alpha-|\varkappa|}
	\]
	for all multi-indexes $\varkappa$ with $|\varkappa|\le N$, $x, y, \lambda\in\mathbb{R}^d$, $\lambda\neq0$, $s, t\in[0,+\infty)$. Here $\gamma\in(0,1)$ is some constant.
	\item\label{A3} In the representing 
	\[
	A(t,x)f(x)=\int_{\mathbb{R}^d}\Omega\left(t,x,\frac{h}{|h|}\right)\frac{f(x)-2f(x-h)+f(x-2h)}{|h|^{d+\alpha}}\,\mathrm{d}h,
	\]
	the function $\Omega(t,x,\cdot)$ is even and non-negative. 
\end{enumerate}
\begin{remark}
	Assumptions \ref{A1}, \ref{A2}, \ref{A3} coincide with assumptions ($A_{41}$), ($A_{42}$), ($A_{44}$) from \cite{Kochubey}, respectively (see pages 266 and 294 there).
\end{remark}

Theorem 4.3 from \cite{Kochubey} (see also Theorem 4.1 there) states that there exists a bounded non-terminating strict Markov process $(x(t))_{t\ge0}$ without second kind discontinuities and the fundamental solution $(g(s,x,t,y))_{0\le s<t, x, y\in\mathbb{R}^d}$ to the equation
\[
\frac{\partial}{\partial t}u(t,x)+A(t,x)u(t,\cdot)(x)=0,\quad t>0,\ x\in\mathbb{R}^d
\] 
is its transition probability density. The function $g$ can be constructed by the parametrix method (see \cite[Sec. 4.1.4]{Kochubey}). 

If the function $a$ is defined by the equality
\[
a(t,x,\lambda)=c|\lambda|^\alpha, \quad t\ge0,\ x\in\mathbb{R}^d,\ \lambda\in\mathbb{R}^d,
\]
where $c>0$ is some constant, the corresponding Markov process is an isotropic $\alpha$-stable process. The function $g$ can be presented by the equality
\[
g(s,x,t,y)=\frac{1}{(2\pi)^d}\int_{\mathbb{R}^d}\exp\{i(x-y,\lambda)-c(t-s)|\lambda|^\alpha\}\mathrm{d}\lambda
\]
in this case. The operator $A$ (it does not depend on $t$ and $x$) is the generator of this process.
This is the simplest example of the processes considered here. Therefore we name the process $(x(t))_{t\ge0}$ by \textit{the $\alpha$-stable type process}. 

Choosing some $0 < \beta \le1$, let us denote the $\beta$-order ``pseudo-gradient'' by $\nabla_\beta$, i.e., it is the operator which is defined by the symbol $(i\lambda|\lambda|^{\beta-1})_{\lambda \in \mathbb{R}^d}$. Note that $\nabla_\beta$ is a vector-valued operator.  In the case of $\beta = 1$, it is the ordinary gradient. We restrict ourselves to case $\beta\neq 1$ in this paper.

Our goal is to construct a perturbation of the considered process using the operator $\left(b, \nabla_ { \beta}\right)$, where $b$ is some measurable $\mathbb{R}^d$-valued function. Under the perturbation of the Markov process $(x(t))_{t\ge0}$ we understand
the construction of a two parametric family of operators $(\mathds{T}_{st})_{0\le s<t}$, defined on the set $C_{b}(\mathbb{R}^{d})$ of real-valued continuous bounded functions defined on $\mathbb{R}^{d}$, such that for every $\varphi\in C_{b}(\mathbb{R}^{d})$ and $t>0$, the function $u(s,x,t)=\mathds{T}_{st}\varphi(x)$ satisfies in some sense
the following Cauchy problem
\[
\frac{\partial}{\partial s}u(s,x,t)+L(s,x)u(s,\cdot,t)(x)=0,\quad 0\le s<t,\ x\in\mathbb{R}^d,
\]
\[
\lim_{s\uparrow t}u(s,x,t)=\varphi(x),\quad x\in\mathbb{R}^d,
\]
where $L(s,x)=A(s,x)+(b(s,x),\nabla_{\beta})$.

The problem of perturbation of Markov processes was and remains in the focus of attention of many researchers. The diffusion case, that is, if $\alpha=2$, was considered by M.~I.~Portenko (see, \cite{Portenko2, Portenko1} and the references there). The perturbation operators there is of the form $(b,\nabla)$, where the function $b$ belongs to some $L_p$-space of functions or is some generalized function of the delta function type. In \cite{Zhang}, the parabolic equation of the type $\nabla(A\nabla u)+B\nabla u-u_t=0$ was considered. The case $\alpha<2$ was considered in the works of K.~Bogdan and T.~Jakubowski \cite{Bogdan, Jakubowski}, who studied such a perturbation with the function $b$ from a Kato class.   S.~I.~Podolynny, M.~I.~Portenko \cite{PodolPort} and M.~I.~Portenko  \cite{Portenko3, Portenko5, Portenko4} investigated this perturbation with the function $b$ from $L_p$. In \cite{Jakubowski1}, $\alpha$-stable process was perturbed by the gradient operator with multiplier $b$ satisfying a certain integral space-time condition. Y.~Maekawa and H.~Miura were concerned with non-local parabolic equations in the presence of a divergence free drift term in \cite{MAEKAWA2013123}. J.-U.~Loebus and M.~I.~Portenko \cite{LobPort} perturbed the infinitesimal operator of a one-dimensional symmetric $\alpha$-stable process using the operator $(q\delta_0,\partial_{\alpha-1})$, where $\partial_{\alpha-1}$ is some pseudo-differential operator of the order $\alpha-1$. Perturbation of $\alpha$-stable process by an operator of fractional Laplacian was considered in \cite{ZhenWang}.
The results of $\alpha \in (1,2)$ and perturbation operators of the type $(b,\nabla_{\alpha - 1})$ can be found in \cite{Osypchuk2, Osypchuk1}. We studied the case of an $\alpha$-stable process and perturbation operators $(b,\nabla_{\beta})$ with a $\mathbb{R}^d$-valued time independent function $b$ from $L_p(\mathbb{R}^d)$ and $0<\beta<\alpha$ in \cite{OsBoiko}. 

This article is structured as follows. The next section is devoted to some auxiliary facts. The perturbation equation is solved in Section \ref{BO_sec3}. In Section \ref{BO_sec4}, we study some properties of the corresponding two-parameter evolutionary family of operators. The last section is devoted to constructing a solution to a corresponding Cauchy problem.

We will not use different notations for constants, if it is not important, but we will use the notations $C$. If we need to emphasize the dependence of the constant $C$ on parameters $\pi$, we will write $C_\pi$.

\section{Auxiliaries}\label{BO_sec2}
First of all let us note that if $0<\beta<1$ the operator $\nabla_{\beta}$ can be represented in the following integral form:
\begin{equation}\label{BO_eq1}
	\nabla_{\beta}f(x)=n_\beta\int_{\mathbb{R}^d} (f(x+y)-f(x))\frac{y}{|y|^{d+\beta+1}}\,\mathrm{d}y,
\end{equation}
where $n_\beta=-2^{-1}\pi^{-(d-1)/2}\Gamma\left(-\beta/2\right) \Gamma\left((d+\beta+1)/2\right)/\Gamma(-\beta)$ (here one have to use the equality $\Gamma(1+x)=x\Gamma(x)$ to expand the Euler gamma function $\Gamma$ to negative non-integer arguments). Representation \eqref{BO_eq1} is true if a function $f$ is at least Lipschitz continuous and bounded.  To obtain the normalizing factor $n_\beta$, it is sufficient to apply the operator $\nabla_\beta$ to the function $f_\lambda(x)=e^{i(\lambda,x)}$, $x\in\mathbb{R}^d$ for any $\lambda\in\mathbb{R}^d$. Moreover this is the way to prove \eqref{BO_eq1}.

Next, we will need some auxiliary statements. The following lemma proved in \cite{Kochubey} (see Lemma 1.11 there). It will be used frequently in this article.
\begin{lemma}\label{BO_lem1}
	The inequality ($0\le s<t$, $x, y\in\mathbb{R}^d$, remind that $1<\alpha<2$)
	\begin{align}
		\int_s^t\mathrm{d}\tau\int_{\mathbb{R}^d}\frac{(\tau-s)^{\lambda/\alpha}}{((\tau-s)^{1/\alpha}+|z-x|)^{d+l}}\frac{(t-\tau)^{\varkappa/\alpha}}{((t-\tau)^{1/\alpha}+|y-z|)^{d+k}}\,\mathrm{d}z\le \nonumber\\
		C\left[B\left(1+\frac{\varkappa-k}{\alpha},1+\frac{\lambda}{\alpha}\right)\frac{(t-s)^{1+(\varkappa+\lambda-k)/\alpha}}{((t-s)^{1/\alpha}+|y-x|)^{d+l}}\right.+ \nonumber\\
		\left.B\left(1+\frac{\varkappa}{\alpha},1+\frac{\lambda-l}{\alpha}\right)\frac{(t-s)^{1+(\varkappa+\lambda-l)/\alpha}}{((t-s)^{1/\alpha}+|y-x|)^{d+k}}\right],	\label{BO_eq2}
	\end{align}
	holds with some constant $C>0$ that depends only on $d$, $\alpha$, $k$ and $l$ for all $\varkappa$, $\lambda$, $k$, $l$, satisfying the inequalities:  $0<k<\alpha+\varkappa$, $0<l<\alpha+\lambda$. Here $B(\cdot,\cdot)$ is the Euler beta function.	
\end{lemma}

The statement of the following lemma can be obtained from the results of \cite[Ch. 4]{Kochubey}).
\begin{lemma}\label{BO_lem2}
	If the assumptions \ref{A1}~--- \ref{A3} hold the transition probability density of the process $(x(t))_{t\ge0}$ have the following properties ($0\le s<t\le T$, $x, y\in\mathbb{R}^d$):
	the function $g$ is continuous differentiable with respect to $x\in\mathbb{R}^d$ and
	\begin{equation}\label{BO_eq3}
		|\partial^k g(s,\cdot,t,y)(x)|\le N_{k,T} \frac{(t-s)^{1-(\gamma+k)/\alpha}}{((t-s)^{1/\alpha}+|y-x|)^{d+\alpha-\gamma}},\quad k=0,1,
	\end{equation}
	where $\partial^k$ means some derivative of the order $k$;	
	\begin{align}
		|\nabla_\beta g(s,\cdot,t,y)(x)|\le N_{\beta,T}\left(\frac{1}{((t-s)^{1/\alpha}+|y-x|)^{d+\beta}}\right.+\nonumber\\
		\left.\frac{(t-s)^{1-\beta/\alpha}}{((t-s)^{1/\alpha}+|y-x|)^{d+\alpha-\gamma}}\right). \label{BO_eq4}	
	\end{align}
	Here, positive constants $N_{k,T}$ and $N_{\beta,T}$ can be depended on $T$.	
\end{lemma}
\begin{proof}
	Let us note that the function $g$ can be constructed by the parametrix method (see \cite[Th. 4.1]{Kochubey}). That is, it can be presented by the following equality
	\[
	g(s,x,t,y)=g_0(s,x,t,y)+h(s,x,t,y),
	\]
	where 
	\[
	g_0(s,x,t,y)=\frac{1}{(2\pi)^d}\int_{\mathbb{R}^d}\exp\{i(x-y,\lambda)-a(t,y,\lambda)(t-s)\}\mathrm{d}\lambda,
	\]
	\begin{equation}\label{BO_eq5}
		h(s,x,t,y)=\int_s^t\mathrm{d}\tau\int_{\mathbb{R}^d}g_0(s,x,\tau,z)\Phi(\tau,z,t,y)\mathrm{d}z,
	\end{equation}
	and the function $\Phi$ is determined by some integral equation.
	
	Theorem 4.1 cited above states that the following inequalities hold
	\[
	|h(s,x,t,y)|\le C\left(\frac{(t-s)^{1+\gamma/\alpha}}{((t-s)^{1/\alpha}+|y-x|)^{d+\alpha}}+\frac{t-s}{((t-s)^{1/\alpha}+|y-x|)^{d+\alpha-\gamma}}\right),
	\]
	\begin{equation}\label{BO_eq6}
		|\Phi(s,x,t,y)|\le C\frac{1}{((t-s)^{1/\alpha}+|y-x|)^{d+\alpha-\gamma}}.
	\end{equation}
	The function $g_0$ satisfies the inequality (see \cite[eq. (4.1.25)]{Kochubey})
	\begin{equation}\label{BO_eq7}
		|\partial^k g_0(s,\cdot,t,y)(x)|\le N_k \frac{t-s}{((t-s)^{1/\alpha}+|y-x|)^{d+\alpha+k}}
	\end{equation}
	in which $\partial^k$ means some derivative of the integer order $0\le k\le N-2d-1$ (the constant $N$ is defined in assumption \ref{A2}), $N_k>0$ are some constants. Note that $k$ can be at least 0, 1, or 2.
	
	Using \eqref{BO_eq2} we can obtain the following inequality valid for $k=0$ or $1$
	\begin{align}
		\int_s^t\mathrm{d}\tau\int_{\mathbb{R}^d}\left|\partial^k g_0(s,\cdot,\tau,z)(x)\Phi(\tau,z,t,y)\right|\mathrm{d}z\le \nonumber\\
		C_k 	\int_s^t\mathrm{d}\tau\int_{\mathbb{R}^d}\frac{\tau-s}{((\tau-s)^{1/\alpha}+|z-x|)^{d+\alpha+k}}\frac{1}{((t-\tau)^{1/\alpha}+|y-z|)^{d+\alpha-\gamma}}\,\mathrm{d}z\le\nonumber\\
		C_k \left(\frac{(t-s)^{1-k/\alpha}}{((t-s)^{1/\alpha}+|y-x|)^{d+\alpha+k}}+\frac{(t-s)^{1+\gamma/\alpha}}{((t-s)^{1/\alpha}+|y-x|)^{d+\alpha-\gamma}}\right)\le  \nonumber\\
		C_{k} \frac{(t-s)^{1-k/\alpha}}{((t-s)^{1/\alpha}+|y-x|)^{d+\alpha-\gamma}},\nonumber
	\end{align}
	where $C_k>0$ are some constants. Therefore (use \eqref{BO_eq7})
	\begin{align}
		|\partial^k g(s,\cdot,t,y)(x)|\le \max(N_k,C_k)\frac{(t-s)^{1-k/\alpha}}{((t-s)^{1-(k+\gamma)/\alpha}+|y-x|)^{d+\alpha+-\gamma}}\times  \nonumber\\ \left[\frac{(t-s)^{(k+\gamma)/\alpha}}{((t-s)^{1/\alpha}+|y-x|)^{k+\gamma}}+(t-s)^{\gamma/\alpha}\right]\le \nonumber\\
		C_{k,T}\frac{(t-s)^{1-(\gamma+k)/\alpha}}{((t-s)^{1/\alpha}+|y-x|)^{d+\alpha-\gamma}}\nonumber,
	\end{align}
	for all $0\le s<t\le T$, $x, y\in\mathbb{R}^d$ and every $T>0$. Here the constants $C_{k,T}>0$ depend on $T$.
	
	For proving inequality \eqref{BO_eq4}, we use representation \eqref{BO_eq1} of the operator $\nabla_\beta$ (remind that $0<\beta<1$). So, since $a(t,x,\cdot)$ is a homogeneous function and 
	\begin{align}
		\nabla_\beta g_0(s,\cdot,t,y)(x)= \nonumber\\
		\frac{i}{(2\pi)^d}\int_{\mathbb{R}^d}\lambda|\lambda|^{\beta-1}\exp\{i(x-y,\lambda)-a(t,y,(t-s)^{1/\alpha}\lambda)\}\mathrm{d}\lambda= \nonumber\\
		\frac{i}{(2\pi)^d(t-s)^{(d+\beta)/\alpha}}\int_{\mathbb{R}^d}\lambda|\lambda|^{\beta-1}\exp\left\{i\left(\frac{x-y}{(t-s)^{1/\alpha}},\lambda\right)-a(t,y,\lambda)\right\}\mathrm{d}\lambda,\nonumber
	\end{align}
	the result of \cite[Lemma 4.2]{Kochubey} leads us to the inequality
	\begin{equation}
		\label{BO_eq8}
		|\nabla_\beta g_0(s,\cdot,t,y)(x)|\le C_\beta\frac{1}{((t-s)^{1/\alpha}+|y-x|)^{d+\beta}}
	\end{equation}
	valid for all $0\le s<t$, $x, y\in\mathbb{R}^d$ with some positive constant $C_\beta$.
	Inequalities \eqref{BO_eq7} allow us to state that the integral
	\[
	\int_{\mathbb{R}^d}\frac{u}{|u|^{d+\beta+1}}\mathrm{d}u\int_s^t\mathrm{d}\tau\int_{\mathbb{R}^d}(g_0(s,x+u,\tau,z)-g_0(s,x,\tau,z))\Phi(\tau,z,t,y)\,\mathrm{d}z
	\] 
	is absolutely convergent. Indeed, for $(\tau,z,u)\in(s,t)\times\mathbb{R}^d\times B_\varepsilon(0)$ and fixed $0\le s<t$, $x, y\in\mathbb{R}^d$ ($\varepsilon>0$ is small enough constant and by $B_\varepsilon(0)$ we denote the ball of the radius $\varepsilon$ and the center placed in the point $0$), we have
	\begin{align}
		R=\left|\frac{u}{|u|^{d+\beta+1}}(g_0(s,x+u,\tau,z)-g_0(s,x,\tau,z))\Phi(\tau,z,t,y)\right|\le\nonumber\\
		\frac{1}{|u|^{d+\beta-1}} \frac{N_1(\tau-s)}{((\tau-s)^{1/\alpha}+|z-x-\theta u|)^{d+\alpha+1}} \frac{C}{((t-\tau)^{1/\alpha}+|y-z|)^{d+\alpha-\gamma}},\nonumber
	\end{align}
	where $\theta=\theta(\tau,z)\in(0,1)$; if $(\tau,z,u)\in(s,t)\times\mathbb{R}^d\times B_\varepsilon(0)^C$, we have the inequality
	\begin{align}
		R\le\frac{N_0}{|u|^{d+\beta}} \left(\frac{\tau-s}{((\tau-s)^{1/\alpha}+|z-x-u|)^{d+\alpha}}+ \frac{\tau-s}{((\tau-s)^{1/\alpha}+|z-x|)^{d+\alpha}}\right)\times\nonumber\\
		\frac{C}{((t-\tau)^{1/\alpha}+|y-z|)^{d+\alpha-\gamma}}.\nonumber
	\end{align}
	The right-hand sides of these inequalities are integrable as the function of $(\tau,z,u)$ on the corresponding domains for every fixed $0\le s<t$ and $x,y\in\mathbb{R}^d$. Using the Fubini's theorem we can write down the equality
	\[
	\nabla_\beta h(s,\cdot,t,y)(x)=\int_s^t\mathrm{d}\tau\int_{\mathbb{R}^d}\nabla_\beta g_0(s,\cdot,\tau,z)(x)\Phi(\tau,z,t,y)\mathrm{d}z,
	\]
	valid for all $0\le s<t$, $x,y\in\mathbb{R}^d$.
	
	Now, use inequalities \eqref{BO_eq6}, \eqref{BO_eq8} and Lemma \ref{BO_lem1}. The result is as follows
	\begin{align}
		|\nabla_\beta h(s,\cdot,t,y)(x)|\le C\left(\frac{(t-s)^{\gamma/\alpha}}{((t-s)^{1/\alpha}+|y-x|)^{d+\beta}}\right.\nonumber\\
		+\left.\frac{(t-s)^{1-\beta/\alpha}}{((t-s)^{1/\alpha}+|y-x|)^{d+\alpha-\gamma}}\right),\nonumber
	\end{align}
	where $C$ is some positive constant. Combining this inequality with \eqref{BO_eq8} we obtain \eqref{BO_eq4}, where it is taken into account that $0\le s<t\le T$. The lemma is proved.
\end{proof}

\section{Perturbation equation solving}
\label{BO_sec3}
Let us consider the perturbation equation ($0\le s<t$, $x, y\in\mathbb{R}^d$)
\begin{equation}
	\label{BO_eq9}
	G(s,x,t,y)=g(s,x,t,y)+\int_s^t\mathrm{d}\tau \int_{\mathbb{R}^d} g(s,x,\tau,z)(b(\tau,z),\nabla_{\beta}G(\tau,\cdot,t,y)(z))\,\mathrm{d}z.
\end{equation}
Using the formal application of the operator $\nabla_{\beta}$ to both sides of equation \eqref{BO_eq9} with respect to the variable $x$ we can search for its solution in the form
\begin{equation}
	\label{BO_eq10}
	G(s,x,t,y)=g(s,x,t,y)+\int_s^t\mathrm{d}\tau \int_{\mathbb{R}^d} g(s,x,\tau,z)(b(\tau,z),v(\tau,z,t,y))\,\mathrm{d}z,
\end{equation}
where the function $(v(s,x,t,y))_{0\le s<t, x, y\in\mathbb{R}^d}$ is a solution to the following equation
\begin{equation}
	\label{BO_eq11}
	v(s,x,t,y)=v_0(s,x,t,y)+\int_s^t\mathrm{d}\tau \int_{\mathbb{R}^d} v_0(s,x,\tau,z)(b(\tau,z),v(\tau,z,t,y))\,\mathrm{d}z,
\end{equation}
in which
\begin{equation}
	\label{BO_eq12}
	v_0(s,x,t,y)=\nabla_{\beta}g(s,\cdot,t,y)(x).
\end{equation}

\begin{thm}\label{BO_th1}
	Let assumptions \ref{A1}~--- \ref{A3} hold and the $\mathbb{R}^d$-valued function $(b(t,x))_{t>0,x\in\mathbb{R}^d}$ belongs to $L_p([0,T]\times\mathbb{R}^d)$ for every $T>0$ with some $p>\frac{d+\alpha}{\alpha-1}$, including $p=+\infty$. Then the following statements are true: there exists a unique solution to equation \eqref{BO_eq11} in the class of functions satisfying the estimate
	\begin{equation}
		\label{BO_eq13}
		|v(s,x,t,y)|\le C_T\left(\frac{1}{((t-s)^{1/\alpha}+|x-y|)^{d+\beta}}+ \frac{(t-s)^{1-\beta/\alpha}}{((t-s)^{1/\alpha}+|x-y|)^{d+\alpha-\gamma}}\right);
	\end{equation}
	the function \eqref{BO_eq10} is a solution to equation \eqref{BO_eq9} and satisfies the estimate
	\begin{align}
		\label{BO_eq14}
		\nonumber	|G(s,x,t,y)|\le C_T\left(\frac{(t-s)^{\beta/\alpha}}{((t-s)^{1/\alpha}+|x-y|)^{d+\beta}}+\right.\\
		\left.\frac{(t-s)^{1-\gamma/\alpha}}{((t-s)^{1/\alpha}+|x-y|)^{d+\alpha-\gamma}}\right);
	\end{align}
	for all $0\le s<t\le T$, $x, y\in\mathbb{R}^d$ and every $T>0$, where a positive constant $C_T$ can be depended on $T$.
\end{thm}

Before proving this theorem, let us prove the following auxiliary statement.
\begin{lemma}\label{BO_lem3}
	If the assumptions of Theorem \ref{BO_th1} hold and a $\mathbb{R}^d$-valued function $(v(s,x,t,y))_{0\le s<t,x, y\in\mathbb{R}^d}$ satisfies inequality \eqref{BO_eq13} then the equality
	\begin{align}
		\nabla_\beta \left(\int_s^t\mathrm{d}\tau\int_{\mathbb{R}^d}g(s,\cdot,\tau,z)(b(\tau,z),v(\tau,z,t,y))\,\mathrm{d}z\right)(x)= \nonumber\\ 
		\int_s^t\mathrm{d}\tau \int_{\mathbb{R}^d}\nabla_\beta g(s,\cdot,\tau,z)(x)(b(\tau,z),v(\tau,z,t,y))\,\mathrm{d}z\label{BO_eq15}
	\end{align}
	is fulfilled.	
\end{lemma}
\begin{proof} Note that the integral
	\[
	\int_s^t\mathrm{d}\tau\int_{\mathbb{R}^d}\nabla_\beta g(s,\cdot,\tau,z)(x)(b(\tau,z),v(\tau,z,t,y))\,\mathrm{d}z
	\] 
	converges (it is sufficiently to use inequalities \eqref{BO_eq4}, \eqref{BO_eq13} and Lemma \ref{BO_lem1}). Indeed, for fixed $0\le s<t\le T$, $x, y\in\mathbb{R}^d$, we can write the following\footnote{Here and further, we use the obvious inequality $(f+g)^r\le 2^r(f^r+g^r)$, valid for all $r\ge0$ and non-negative functions $f$ and $g$, without any references.}
	\begin{align}
		\int_s^t\mathrm{d}\tau\int_{\mathbb{R}^d}|\nabla_\beta g(s,\cdot,\tau,z)(x)||b(\tau,z)||v(\tau,z,t,y)|\,\mathrm{d}z\le\nonumber\\
		\|b\|_p^TC_TN_{\beta,T}\times\nonumber\\
		\left(\int_s^t\mathrm{d}\tau\int_{\mathbb{R}^d} \left(\frac{1}{((\tau-s)^{1/\alpha}+|x-z|)^{d+\beta}}+ \frac{(\tau-s)^{1-\beta/\alpha}}{((\tau-s)^{1/\alpha}+|x-z|)^{d+\alpha-\gamma}}\right)^q\times\right.\nonumber\\
		\left.\left(\frac{1}{((t-\tau)^{1/\alpha}+|z-y|)^{d+\beta}}+ \frac{(t-\tau)^{1-\beta/\alpha}}{((t-\tau)^{1/\alpha}+|z-y|)^{d+\alpha-\gamma}}\right)^q\mathrm{d}z\right)^{1/q}	\le \nonumber\\
		C_T \left(
		\frac{(t-s)^{1/q-d/(\alpha p)-\beta/\alpha}}{((t-s)^{1/\alpha}+|x-y|)^{d+\beta}}+
		\frac{(t-s)^{1/q-d/(\alpha p)+(\gamma-\beta)/\alpha}}{((t-s)^{1/\alpha}+|x-y|)^{d+\beta}}+\right.\nonumber\\
		\left.
		\frac{(t-s)^{1+1/q-d/(\alpha p)-2\beta/\alpha}}{((t-s)^{1/\alpha}+|x-y|)^{d+\alpha-\gamma}}+ 
		\frac{(t-s)^{1+1/q-d/(\alpha p)-(2\beta-\gamma)/\alpha}}{((t-s)^{1/\alpha}+|x-y|)^{d+\alpha-\gamma}}
		\right)\le\nonumber\\ 
		C_T (t-s)^{1/q-d(1/p+1)/\alpha-2\beta/\alpha}.\nonumber
	\end{align}
	
	To prove the existence of left-hand side of \eqref{BO_eq15} we use representation \eqref{BO_eq1} of the operator $\nabla_\beta$. For fixed $0\le s<t\le T$ and $y\in\mathbb{R}^d$ the function 
	\[
	\left(\int_s^t\mathrm{d}\tau\int_{\mathbb{R}^d}g(s,x,\tau,z)(b(\tau,z),v(\tau,z,t,y))\,\mathrm{d}z \right)_{x\in\mathbb{R}^d}
	\]
	is bounded and Lipschitz continuous. Indeed, it is sufficient to prove a boundedness of the integral 
	\[
	\int_s^t\mathrm{d}\tau\int_{\mathbb{R}^d}\nabla_kg(s,\cdot,\tau,z)(x)(b(\tau,z),v(\tau,z,t,y))\,\mathrm{d}z,
	\]
	as the function of $x\in\mathbb{R}^d$ for fixed $0\le s<t\le T$ and $y\in\mathbb{R}^d$ in the cases of $\nabla_0=I$ (the identical operator) and $\nabla_1=\nabla$ (the gradient). 
	
	Using inequalities \eqref{BO_eq3}, \eqref{BO_eq13} and Lemma \ref{BO_lem1}, for all $x\in\mathbb{R}^d$, we can write down ($k=0,1$)
	\begin{align}
		\int_s^t\mathrm{d}\tau\int_{\mathbb{R}^d}|\nabla_k g(s,\cdot,\tau,z)(x)||b(\tau,z)||v(\tau,z,t,y)|\,\mathrm{d}z\le\nonumber\\
		\|b\|_p^TC_TN_{k,T}\left(\int_s^t\mathrm{d}\tau\int_{\mathbb{R}^d} \frac{(\tau-s)^{(1-(\gamma+k)/\alpha)q}}{((\tau-s)^{1/\alpha}+|x-z|)^{(d+\alpha-\gamma)q}}\times\right.\nonumber\\
		\left.\left(\frac{1}{((t-\tau)^{1/\alpha}+|z-y|)^{d+\beta}}+ \frac{(t-\tau)^{1-\beta/\alpha}}{((t-\tau)^{1/\alpha}+|z-y|)^{d+\alpha-\gamma}}\right)^q\mathrm{d}z\right)^{1/q}	\le \nonumber\\
		C_T \left(
		\frac{(t-s)^{1+1/q-d/(\alpha p)-(\gamma+\beta+k)/\alpha}}{((t-s)^{1/\alpha}+|x-y|)^{d+\alpha-\gamma}}+
		\frac{(t-s)^{1/q-d/(\alpha p)-k/\alpha}}{((t-s)^{1/\alpha}+|x-y|)^{d+\beta}}+\right.\nonumber\\
		\left.
		\frac{(t-s)^{1+1/q-d/(\alpha p)-(\beta+k)/\alpha}}{((t-s)^{1/\alpha}+|x-y|)^{d+\alpha-\gamma}}
		\right)\le C_T (t-s)^{1/q-d(1/p+1)/\alpha-(\beta+k)/\alpha},\nonumber
	\end{align}
	which is bounded for fixed $0\le s<t\le T$.
	
	Together with the operator $\nabla_{\beta}$, we consider a family of operators $\{\nabla_\beta^{(\varepsilon)}:\varepsilon>0\}$ acting on a function $(f(x))_{x\in\mathbb{R}^d}$ according to the following rule
	\[
	\nabla_\beta^{(\varepsilon)}\varphi(x)=n_\beta\int_{\varepsilon\le|y|\le\frac{1}{\varepsilon}} (f(x+y)-f(x))\frac{y}{|y|^{d+\beta+1}}\,\mathrm{d}y.
	\]
	It is clear that the following relation $\lim_{\varepsilon\to0+}\nabla_\beta^{(\varepsilon)}f(x)=\nabla_\beta f(x)$, $x\in\mathbb{R}^d$ holds for every bounded and Lipschitz continuous function $f$.
	
	Using inequalities \eqref{BO_eq3}, \eqref{BO_eq13} we can state, that the following inequality 
	\begin{align}
		\left|\frac{u}{|u|^{d+\beta+1}} (g(s,x+u,\tau,z)-g(s,x,\tau,z))(b(\tau,z),v(\tau,z,t,y))\right|\le \nonumber\\
		\frac{{N_{0,T}} C_T|b(\tau,z)|}{|u|^{d+\beta}}\times\nonumber\\
		\left(\frac{(\tau-s)^{1-\gamma/\alpha}}{((\tau-s)^{1/\alpha}+|z-x-u|)^{d+\alpha-\gamma}}+ \frac{(\tau-s)^{1-\gamma/\alpha}}{((\tau-s)^{1/\alpha}+|z-x|)^{d+\alpha-\gamma}}\right)\times\nonumber\\
		\left(\frac{1}{((t-\tau)^{1/\alpha}+|y-z|)^{d+\beta}}+ \frac{(t-\tau)^{1-\beta/\alpha}}{((t-\tau)^{1/\alpha}+|y-z|)^{d+\alpha-\gamma}}\right)
		\nonumber
	\end{align}
	is true for every $T>0$ and all $0\le s<\tau<t\le T$, $x, y, z, u\in\mathbb{R}^d$.
	
	Using H\"{o}lder's inequality, we can write down the following
	\begin{align}
		\int_s^t\hskip-5pt\mathrm{d}\tau\hskip-5pt\int_{\mathbb{R}^d}\left|(g(s,x+u,\tau,z)-g(s,x,\tau,z))(b(\tau,z),v(\tau,z,t,y))\right|\,\mathrm{d}z\le \nonumber\\
		\le 4C_T{N_{0,T}}\|b\|_p^T \nonumber\\
		\left[\int_s^t\hskip-5pt\mathrm{d}\tau\hskip-5pt\int_{\mathbb{R}^d}\frac{(\tau-s)^{(1-\gamma/\alpha)q}}{((\tau-s)^{1/\alpha}+|z-x-u|)^{(d+\alpha-\gamma)q}}
		\frac{1}{((t-\tau)^{1/\alpha}+|y-z|)^{(d+\beta)q}}\,\mathrm{d}z+\right. \nonumber\\
		\left.
		\int_s^t\hskip-5pt\mathrm{d}\tau\hskip-5pt\int_{\mathbb{R}^d}\frac{(\tau-s)^{(1-\gamma/\alpha)q}}{((\tau-s)^{1/\alpha}+|z-x|)^{(d+\alpha-\gamma)q}}
		\frac{1}{((t-\tau)^{1/\alpha}+|y-z|)^{(d+\beta)q}}\,\mathrm{d}z+\right.\nonumber\\
		\left.\int_s^t\hskip-5pt\mathrm{d}\tau\hskip-5pt\int_{\mathbb{R}^d}\frac{(\tau-s)^{(1-\gamma/\alpha)q}}{((\tau-s)^{1/\alpha}+|z-x-u|)^{(d+\alpha-\gamma)q}}
		\frac{(t-\tau)^{(1-\beta/\alpha)q}}{((t-\tau)^{1/\alpha}+|y-z|)^{(d+\alpha-\gamma)q}}\,\mathrm{d}z+\right. \nonumber\\
		\left.
		\int_s^t\hskip-5pt\mathrm{d}\tau\hskip-5pt\int_{\mathbb{R}^d}\frac{(\tau-s)^{(1-\gamma/\alpha)q}}{((\tau-s)^{1/\alpha}+|z-x|)^{(d+\alpha-\gamma)q}}
		\frac{(t-\tau)^{(1-\beta/\alpha)q}}{((t-\tau)^{1/\alpha}+|y-z|)^{(d+\alpha-\gamma)q}}\,\mathrm{d}z\right]^{1/q}, \nonumber
	\end{align}
	which holds for all $0\le s<t\le T$, $x, y, u\in\mathbb{R}^d$. Here the number $q\ge1$ is defined by the relation 
	$\frac{1}{p}+\frac{1}{q}=1$ (if $p=+\infty$ then $q=1$) and the norm of the function $b$ in the space $L_p([0,T]\times\mathbb{R}^d)$ is denoted by $\|b\|_p^T$.
	It is easy to see, that the right hand side integrals satisfy the conditions of Lemma \ref{BO_lem1}. Therefore these integrals are finite for all fixed $0\le s<t\le T$, $x, y, u\in\mathbb{R}^d$.
	As a consequence, the function
	\[
	\left((g(s,x+u,\tau,z)-g(s,x,\tau,z))(b(\tau,z),v(\tau,z,t,y))\frac{u}{|u|^{d+\beta+1}}\right)_{\tau\in(s,t),u,z\in\mathbb{R}^d}
	\]
	is integrable with respect to $(u,\tau,z)\in\{\varepsilon\le|u|\le\frac{1}{\varepsilon}\}\times(s,t)\times\mathbb{R}^d$ with small enough $\varepsilon>0$ for every fixed $0\le s<t\le T$, $x, y\in\mathbb{R}^d$. 
	Thus, using Fubini's theorem, we obtain the following equality
	\begin{align}
		\nabla_\beta^{(\varepsilon)}\left( \int_s^t\mathrm{d}\tau\int_{\mathbb{R}^d}g(s,\cdot,\tau,z)(b(\tau,z),v(\tau,z,t,y))\,\mathrm{d}z\right)(x)= \nonumber\\
		\int_s^t\mathrm{d}\tau\int_{\mathbb{R}^d}\nabla_\beta^{(\varepsilon)} g(s,\cdot,\tau,z)(x)(b(\tau,z),v(\tau,z,t,y))\,\mathrm{d}z.\label{BO_eq16}
	\end{align}
	hold for each $\varepsilon>0$.
	
	So, it is enough to pass to the limit as $\varepsilon\to0+$ in equality \eqref{BO_eq16} to complete the proof of the lemma.
\end{proof}

\begin{proof}[Proof of Theorem \ref{BO_th1}]
	Let us solve equation \eqref{BO_eq11} for all fixed $0\le s<t\le T$, $x, y\in\mathbb{R}^d$, using the method of successive approximations. Namely, consider the sequence of functions $(v_{k}(s,x,t,y))_{0\le s<t\le T, x, y\in\mathbb{R}^{d}}$, $k=0,1,2,\dots$, 
	given by the recurrence relation
	\begin{equation*}
		v_{k+1}(s,x,t,y) =\int_{s}^{t}\mathrm{d} \tau \int_{\mathbb{R}^{d}}v_{0} (s,x,\tau,z)(b(\tau,z), v_{k}(\tau,z,t,y))\, \mathrm{d}z,
	\end{equation*}
	where $v_0$ is defined by \eqref{BO_eq12}.
	
	Using inequality \eqref{BO_eq4} one can obtain the relation ($0\le s<t\le T$, $x, y\in\mathbb{R}^d$)
	\begin{align}\label{BO_eq17}
		|v_{k+1}(s,x,t,y)| \le 2N_{\beta,T}\|b\|_p^T\left(\int_{s}^{t}\mathrm{d} \tau \int_{\mathbb{R}^{d}}\frac{|v_{k}(\tau,z,t,y)|^q}{((\tau-s)^{1/\alpha}+|x-z|)^{(d+\beta)q}}\, \mathrm{d}z+\right.\nonumber\\
		\left.\int_{s}^{t}\mathrm{d} \tau \int_{\mathbb{R}^{d}}\frac{|v_{k}(\tau,z,t,y)|^q(\tau-s)^{(1-\beta/\alpha)q}}{((\tau-s)^{1/\alpha}+|x-z|)^{(d+\alpha-\gamma)q}}\, \mathrm{d}z\right)^{1/q}.
	\end{align}
	Iterating relation \eqref{BO_eq17} by using inequality \eqref{BO_eq2} we come to the estimate
	\begin{align}\label{BO_eq18}
		|v_k(s,x,t,y)|\le R_{k,T}\times \nonumber\\
		\left(\frac{(t-s)^{k\theta}}{((t-s)^{1/\alpha}+|x-y|)^{d+\beta}}+
		\frac{(t-s)^{k\theta+1-\beta/\alpha}}{((t-s)^{1/\alpha}+|x-y|)^{d+\alpha-\gamma}}\right),
	\end{align}
	where $\theta=1-((d+\alpha)/p+\beta)/\alpha$ (note that $\theta >(1-\beta)/\alpha$) and the sequence $\{R_{k,T}:k=0,1,2,\dots\}$ satisfies the relations $R_0=N_{\beta,T}$ and 
	\begin{align}\label{BO_eq19}
		R_{k+1,T}= R_{k,T}8^{1+1/q}N_{\beta,T}\|b\|_p^TC^{1/q}(1+T^{\gamma/\alpha})\times\nonumber\\
		\max\left(B\left((k+1)\theta q,1\right)^{1/q},B\left(1+k\theta q,\theta q\right)^{1/q} \right).
	\end{align}
	for all $k=1,2,\dots$. Here $C$ is the maximum of a finite number of constants taking from inequality \eqref{BO_eq2}.
	Indeed, when $k=0$, inequality \eqref{BO_eq18} coincides with inequality \eqref{BO_eq4}. If \eqref{BO_eq18} is correct for some $k\in\mathbb{N}$ we obtain
	\begin{align}
		|v_{k+1}(s,x,t,y)|\le 8N_{\beta,T}\|b\|_p^TR_{k,T}\times \nonumber\\
		\left[\int_{s}^{t}\mathrm{d} \tau \int_{\mathbb{R}^{d}}\frac{(\tau-s)^{(1-\beta/\alpha)q}}{((\tau-s)^{1/\alpha}+|x-z|)^{(d+\alpha-\gamma)q}}\frac{(t-\tau)^{k\theta q}}{((t-\tau)^{1/\alpha}+|z-y|)^{(d+\beta)q}}\, \mathrm{d}z+\right. \nonumber\\
		\left.\int_{s}^{t}\mathrm{d} \tau \int_{\mathbb{R}^{d}}\frac{(\tau-s)^{(1-\beta/\alpha)q}}{((\tau-s)^{1/\alpha}+|x-z|)^{(d+\alpha-\gamma)q}}\frac{(t-\tau)^{k\theta q+(1-\beta/\alpha)q}}{((t-\tau)^{1/\alpha}+|z-y|)^{(d+\alpha-\gamma)q}}\, \mathrm{d}z+\right. \nonumber\\
		\left.\int_{s}^{t}\mathrm{d} \tau \int_{\mathbb{R}^{d}}\frac{1}{((\tau-s)^{1/\alpha}+|x-z|)^{(d+\beta)q}}\frac{(t-\tau)^{k\theta q+(1-\beta/\alpha)q}}{((t-\tau)^{1/\alpha}+|z-y|)^{(d+\alpha-\gamma)q}}\, \mathrm{d}z+\right.
		\nonumber\\
		\left.\int_{s}^{t}\mathrm{d} \tau \int_{\mathbb{R}^{d}}\frac{1}{((\tau-s)^{1/\alpha}+|x-z|)^{(d+\beta)q}}\frac{(t-\tau)^{k\theta q}}{((t-\tau)^{1/\alpha}+|z-y|)^{(d+\beta)q}}\, \mathrm{d}z
		\right]^{1/q}. 
		\nonumber
	\end{align}
	Using inequality \eqref{BO_eq2}, we can write
	\begin{align}
		|v_{k+1}(s,x,t,y)|\le 8N_{\beta,T}\|b\|_p^TR_{k,T} C^{1/q}\times\nonumber\\
		\left[
		B\left((k+1)\theta q,1+\left(1-\frac{\beta}{\alpha}\right)q\right)
		\frac{(t-s)^{(k+1)\theta q+(1-\beta/\alpha)q}}{((t-s)^{1/\alpha}+|x-y|)^{(d+\alpha-\gamma)q}}+\right.\nonumber\\
		\left.
		B\left(1+k\theta q,\theta q+\frac{\gamma}{\alpha}q\right)
		\frac{(t-s)^{(k+1)\theta q+q\gamma/\alpha}}{((t-s)^{1/\alpha}+|x-y|)^{(d+\beta)q}}+\right.\nonumber\\
		\left.
		B\left((k+1)\theta q+\frac{\gamma}{\alpha}q,1+\left(1-\frac{\beta}{\alpha}\right)q\right)
		\frac{(t-s)^{(k+1)\theta q+(1-\beta/\alpha+\gamma/\alpha)q}}{((t-s)^{1/\alpha}+|x-y|)^{(d+\alpha-\gamma)q}}+\right.\nonumber
	\end{align}
	\begin{align}
		\left.
		B\left(1+k\theta q+\left(1-\frac{\beta}{\alpha}\right),\theta q+\frac{\gamma}{\alpha}q\right)
		\frac{(t-s)^{(k+1)\theta q+(1-\beta/\alpha+\gamma/\alpha)q}}{((t-s)^{1/\alpha}+|x-y|)^{(d+\alpha-\gamma)q}}+\right.\nonumber\\
		\left.
		B\left((k+1)\theta q+\frac{\gamma}{\alpha}q,1\right)
		\frac{(t-s)^{(k+1)\theta q+q\gamma/\alpha}}{((t-s)^{1/\alpha}+|x-y|)^{(d+\beta)q}}+\right.\nonumber\\
		\left.
		B\left(1+k\theta q+\left(1-\frac{\beta}{\alpha}\right)q,\theta q\right)
		\frac{(t-s)^{(k+1)\theta q+(1-\beta/\alpha)q}}{((t-s)^{1/\alpha}+|x-y|)^{(d+\alpha-\gamma)q}}+\right.\nonumber\\
		\left.
		B\left((k+1)\theta q,1\right)
		\frac{(t-s)^{(k+1)\theta q}}{((t-s)^{1/\alpha}+|x-y|)^{(d+\beta)q}}+\right.\nonumber\\
		\left.
		B\left(1+k\theta q,\theta q\right)
		\frac{(t-s)^{(k+1)\theta q}}{((t-s)^{1/\alpha}+|x-y|)^{(d+\beta)q}}
		\right]^{1/q}.  \nonumber
	\end{align}
	Since the function $(B(x,y))_{x>0,y>0}$ decreases with respect to each of its arguments, we have the following:
	\begin{align}
		|v_{k+1}(s,x,t,y)|\le 8N_{\beta,T}\|b\|_p^TR_{k,T} C^{1/q}\times\nonumber\\
		\max\left(B\left((k+1)\theta q,1\right)^{1/q},B\left(1+k\theta q,\theta q\right)^{1/q} \right)8^{1/q}(1+T^{\gamma/\alpha})\times\nonumber\\
		\left(
		\frac{(t-s)^{(k+1)\theta}}{((t-s)^{1/\alpha}+|x-y|)^{d+\beta}}+
		\frac{(t-s)^{(k+1)\theta +1-\beta/\alpha}}{((t-s)^{1/\alpha}+|x-y|)^{d+\alpha-\gamma}} 
		\right)\nonumber
	\end{align}
	that is, \eqref{BO_eq18} holds for $k+1$. Therefore, it is correct for all $k=0,1,2,\dots$ and relation \eqref{BO_eq19} is true.

	Since $\lim_{k\to\infty}B(1+k\theta q,\theta q)=\lim_{k\to\infty}B((k+1)\theta q,1)=0$ the series $\sum_{k=0}^\infty v_k(s,x,t,y)$ converges uniformly with respect to $x, y\in\mathbb{R}^d$ and locally uniformly with respect to $0\le s<t$. Let $(v(s,x,t,y))_{0\le s<t, x, y\in\mathbb{R}^{d}}$ be the sum of this series. The function $v$ is a solution to equation \eqref{BO_eq11} and the following estimate 
	\[
	|v(s,x,t,y)|\le C_T  \left(\frac{1}{((t-s)^{1/\alpha}+|x-y|)^{d+\beta}}+ \frac{(t-s)^{1-\beta/\alpha}}{((t-s)^{1/\alpha}+|x-y|)^{d+\alpha-\gamma}}\right),
	\]
	holds for all $0\le s<t\le T$, $x, y\in\mathbb{R}^d$ and every $T>0$ with some constant $C_T>0$ depended on $T$.
	
	The uniqueness of this solution follows from the fact that the difference ${v^*}$ of every two such solutions satisfies the equation
	\[
	{v^*}(s,x,t,y)=\int_s^t\mathrm{d}\tau \int_{\mathbb{R}^d} v_0(s,x,\tau,z)(b(\tau,z),{v^*}(\tau,z,t,y))\,\mathrm{d}z.
	\]
	Therefore, since ${v^*}$ satisfies estimate \eqref{BO_eq13}, using inequality \eqref{BO_eq2}, one can obtain the following ($0\le s<t\le T$, $x, y\in\mathbb{R}^d$)
	\begin{align}
		|{v^*}(s,x,t,y)|\le  R_{k,T}\times\nonumber\\ \left(\frac{(t-s)^{k\theta}}{((t-s)^{1/\alpha}+|x-y|)^{d+\beta}}+\frac{(t-s)^{k\theta+1-\beta/\alpha}}{((t-s)^{1/\alpha}+|x-y|)^{d+\alpha-\gamma}}\right),\nonumber
	\end{align}
	for all $k\in\mathbb{N}$, $T>0$, where $R_{k,T}$ is defined by \eqref{BO_eq19}. Since $\lim_{k\to\infty}R_{k,T}T^{k\theta}=0$ for all $T>0$, this means that ${v^*}(s,x,t,y)\equiv0$.

	Let us define the function $G$ by equality \eqref{BO_eq10} with the function $v$ just constructed. Using \eqref{BO_eq2}, \eqref{BO_eq3} and H\"{o}lder's inequality we can write the following chain of inequalities
	\begin{align}
		|G(s,x,t,y)|\le N_{0,T}\frac{(t-s)^{1-\gamma/\alpha}}{((t-s)^{1/\alpha}+|x-y|)^{d+\alpha-\gamma}} +\nonumber\\  
		2^{1+1/q}N_{0,T} C_T\|b\|_p^T\times\nonumber\\ 
		\left[\left(\int_s^t\mathrm{d}\tau \int_{\mathbb{R}^d} \frac{(\tau-s)^{(1-\beta/\alpha)q}}{((\tau-s)^{1/\alpha}+|x-z|)^{(d+\alpha-\gamma)q}}  \frac{1}{((t-\tau)^{1/\alpha}+|z-y|)^{(d+\beta)q}}\,\mathrm{d}z\right)^{1/q}+\right.\nonumber\\
		\left.\left(\int_s^t\mathrm{d}\tau \int_{\mathbb{R}^d} \frac{(\tau-s)^{(1-\beta/\alpha)q}}{((\tau-s)^{1/\alpha}+|x-z|)^{(d+\alpha-\gamma)q}}  \frac{(t-\tau)^{(1-\beta/\alpha)q}}{((t-\tau)^{1/\alpha}+|z-y|)^{(d+\beta)q}}\,\mathrm{d}z\right)^{1/q}\right]\le\nonumber	\\
		{C}_T\left[\frac{(t-s)^{\beta/\alpha}}{((t-s)^{1/\alpha}+|x-y|)^{d+\beta}}+\frac{(t-s)^{1-\gamma/\alpha}}{((t-s)^{1/\alpha}+|x-y|)^{d+\alpha-\gamma}}\right]
		,\nonumber
	\end{align}
	where ${C}_T$ (in the last expression) is some positive constant, which might be depended on $T$. 
	
	Let us prove that $\nabla_{\beta}G(s,\cdot,t,y)(x)\equiv v(s,x,t,y)$. 
	Using the Lemma \ref{BO_lem3} statement, we obtain the following equality
	\begin{align}
		\nabla_{\beta}G(s,\cdot,t,y)(x)=\nonumber\\
		v_0(s,x,t,y)+\int_s^t\mathrm{d}\tau \int_{\mathbb{R}^d} v_0(s,x,\tau,z)(b(\tau,z),v(\tau,z,t,y))\,\mathrm{d}z=\nonumber\\
		v(s,x,t,y),\nonumber
	\end{align}
	which is true for all $0\le s<t$, $x, y\in\mathbb{R}^d$.
	
	Consequently the function $G$ is a solution to equation \eqref{BO_eq9} and it satisfies estimate \eqref{BO_eq14}. The theorem is proved.
\end{proof}

\section{Evolution family of operators}
\label{BO_sec4}
Let us define the two-parameter family of operators $\{\mathds{T}_{st}:0\le s<t\}$ defined in the space of continuous bounded functions ${C}_b(\mathbb{R}^d)$ by the equality
\begin{equation}\label{BO_eq20}
	\mathds{T}_{st}\varphi(x)=\int_{\mathbb{R}^d}G(s,x,t,y)\varphi(y)\,\mathrm{d}y,\quad \varphi\in{C}_b(\mathbb{R}^d), \ x\in \mathbb{R}^d.
\end{equation}
Similarly to the proof of Theorem \ref{BO_th1} one can prove the following statement.

\begin{lemma}\label{BO_lem4}
	The function $w(s,x,t,\varphi)=\int_{\mathbb{R}^d}v(s,x,t,y)\varphi(y)\,\mathrm{d}y$, $0\le s<t$, $x\in \mathbb{R}^d$ and $\varphi\in{C}_b(\mathbb{R}^d)$ ($v$ is defined in Theorem \ref{BO_th1}), is a unique (in the class of functions, which satisfy the inequality
	$
	|w(s,x,t,\varphi)|\le C_T(t-s)^{-\beta/\alpha}
	$)
	solution to the equation
	\begin{equation}\label{BO_eq21}
		w(s,x,t,\varphi)=w_0(s,x,t,\varphi)+\int_s^t\mathrm{d}\tau \int_{\mathbb{R}^d} v_0(s,x,\tau,z)(b(\tau,z),w(\tau,z,t,\varphi))\,\mathrm{d}z,
	\end{equation}
	for all $0\le s<t\le T$, $x\in \mathbb{R}^d$ and every $T>0$.
	Here $v_0$ is defined by \eqref{BO_eq12} and $w_0(s,x,t,\varphi)=\int_{\mathbb{R}^d}v_0(s,x,t,y)\varphi(y)\,\mathrm{d}y$.
\end{lemma}
\begin{proof}
	Equation \eqref{BO_eq21} is obtained from equation \eqref{BO_eq11} by multiplying it with the functions $\varphi$ and using the Fubini's theorem. The justification of the usage of the Fubini’s theorem bases on estimates \eqref{BO_eq4}, \eqref{BO_eq13} (see also \eqref{BO_eq12}). Indeed,
	\begin{align}
		\int_{\mathbb{R}^d}\,\mathrm{d}y\int_s^t\mathrm{d}\tau \int_{\mathbb{R}^d} |v_0(s,x,\tau,z)||b(\tau,z)||v(\tau,z,t,y)||\varphi(y)|\,\mathrm{d}z\le  \nonumber \\
		4\|\varphi\|\|b\|_p^TN_{\beta,T}C_T \times \nonumber \\
		\int_{\mathbb{R}^d}\,\mathrm{d}y\left(\int_s^t\mathrm{d}\tau \int_{\mathbb{R}^d} \frac{1}{((\tau-s)^{1/\alpha}+|z-x|)^{(d+\beta)q}}\times
		\right. \nonumber \\
		\left.
		\frac{1}{((t-\tau)^{1/\alpha}+|y-z|)^{(d+\beta)q}}\,\mathrm{d}z +\right. \nonumber \\
		\left.\int_s^t\mathrm{d}\tau \int_{\mathbb{R}^d} \frac{1}{((\tau-s)^{1/\alpha}+|z-x|)^{(d+\beta)q}}\frac{(t-\tau)^{(1-\beta/\alpha)q}}{((t-\tau)^{1/\alpha}+|y-z|)^{(d+\alpha-\gamma)q}}\,\mathrm{d}z +\right. \nonumber \\
		\left.\int_s^t\mathrm{d}\tau \int_{\mathbb{R}^d} \frac{(\tau-s)^{(1-\beta/\alpha)q}}{((\tau-s)^{1/\alpha}+|z-x|)^{(d+\alpha-\gamma)q}}\frac{1}{((t-\tau)^{1/\alpha}+|y-z|)^{q(d+\beta)}}\,\mathrm{d}z +\right. \nonumber \\
		\left.\int_s^t\mathrm{d}\tau \int_{\mathbb{R}^d} \frac{(\tau-s)^{(1-\beta/\alpha)q}}{((\tau-s)^{1/\alpha}+|z-x|)^{(d+\alpha-\gamma)q}}\times
		\right. \nonumber \\
		\left.
		\frac{(t-\tau)^{(1-\beta/\alpha)q}}{((t-\tau)^{1/\alpha}+|y-z|)^{(d+\alpha-\gamma)q}}\,\mathrm{d}z\right)^{1/q}\le   \nonumber \\
		{C}_T \int_{\mathbb{R}^d}\left(\frac{1}{((t-s)^{1/\alpha}+|x-y|)^{d+\beta}}+ \frac{(t-s)^{1-\beta/\alpha}}{((t-s)^{1/\alpha}+|x-y|)^{d+\alpha-\gamma}}\right)\,\mathrm{d}y\le  \nonumber\\
		{C}_T(t-s)^{-\beta/\alpha}  \nonumber
	\end{align}
	for each $T>0$ and every $0\le s<t\le T$, $x\in\mathbb{R}^d$ with some constants ${C}_T>0$. Here we used the well-known formula\footnote{Below we will use this formula without any references.} $\int_{\mathbb{R}^d} (a+|x|)^{-d-\varkappa}\,\mathrm{d}x=a^{-\varkappa}B(d,\varkappa)\frac{2\pi^{d/2}}{\Gamma(d/2)}$, which is valid for all $a>0$ and $\varkappa>0$.
\end{proof}

The next theorem contains the properties of the family of operators \eqref{BO_eq20}.
\begin{thm}
	\label{BO_th2}
	Let the Theorem \ref{BO_th1} assumptions hold. Then the following statements are true:
	\begin{itemize}
		\item the operators $\mathds{T}_{st}$, $0\le s<t$ are linear and bounded on $C_b(\mathbb{R}^d)$;
		\item if $\varphi(x)\equiv1$ then $\mathds{T}_{st}\varphi(x)\equiv1$;
		\item the family of operators $\{\mathds{T}_{st}:0\le s<t\}$ has an evolution property, that is, $\mathds{T}_{s\tau}\mathds{T}_{\tau t}=\mathds{T}_{st}$ for all $0\le s<\tau<t$;
		\item $\wlim_{s\uparrow t}\mathds{T}_{st}=I$, where $I$ is the identical operator, i.e., $\lim_{s\uparrow t}\mathds{T}_{st}\varphi(x)=\varphi(x)$, $x\in\mathbb{R}^d$ for all $\varphi\in C_b(\mathbb{R}^d)$.
	\end{itemize}
	
\end{thm}
\begin{proof}
	The linearity of operator $\mathds{T}_{st}$ is evident. Let us prove its boundedness. If $\varphi\in C_b(\mathbb{R}^d)$ then using inequality \eqref{BO_eq14} we can write ($\|\varphi\|=\max_{x\in\mathbb{R}^d}|\varphi(x)|$)
	\begin{align}
		|\mathds{T}_{st}\varphi(x)|\le\|\varphi\|\int_{\mathbb{R}^d}|G(s,x,t,y)|\,\mathrm{d}y\le\nonumber\\ C_T\|\varphi\|\int_{\mathbb{R}^d}\left(\frac{(t-s)^{\beta/\alpha}}{((t-s)^{1/\alpha}+|x-y|)^{d+\beta}}+\frac{(t-s)^{1-\gamma/\alpha}}{((t-s)^{1/\alpha}+|x-y|)^{d+\alpha-\gamma}}\right)\, \mathrm{d}y\le\nonumber\\
		{C}_T\|\varphi\|\nonumber
	\end{align}
	for all $0\le s<t\le T$ and each $T>0$. Therefore the operators $\mathds{T}_{st}$ are bounded.
	
	Next, if $\varphi(x)\equiv1$ then	
	\begin{align}
		\mathds{T}_{st}\varphi(x)=\int_{\mathbb{R}^d}g(s,x,t,y)\,\mathrm{d}y+\nonumber\\
		\int_s^t\mathrm{d}\tau \int_{\mathbb{R}^d} g(s,x,\tau,z)(b(\tau,z),\int_{\mathbb{R}^d}v(\tau,z,t,y)\,\mathrm{d}y)\,\mathrm{d}z.\nonumber
	\end{align}
	The function $w(s,x,t,1)=\int_{\mathbb{R}^d}v(\tau,z,t,y)\,\mathrm{d}y$, $0\le s<t$, $x\in\mathbb{R}^d$ is a solution to equation \eqref{BO_eq21} with $w_0(s,x,t,1)\equiv0$.
	So, $w(s,x,t,1)\equiv0$ and $\mathds{T}_{st}\varphi(x)=\int_{\mathbb{R}^d}g(s,x,t,y)\,\mathrm{d}y\equiv1$.
	
	Although the evolution property can be proved in a standard way (see, for example \cite{Osypchuk2,Osypchuk1,Portenko2,Portenko3,Portenko5}), we will provide it here. For this, let us choose arbitrary $0\le s<u<t$, $\varphi\in C_b(\mathbb{R}^d)$, $x\in\mathbb{R}^d$ and consider 
	\begin{align}
		\mathds{T}_{st}\varphi(x)=\mathds{T}_{st}^0\varphi(x)+
		\int_s^t\mathrm{d}\tau \int_{\mathbb{R}^d} g(s,x,\tau,z)(b(\tau,z),w(\tau,z,t,\varphi))\,\mathrm{d}z= \nonumber \\
		\mathds{T}_{su}^0(\mathds{T}_{ut}^0\varphi)(x)+
		\int_s^u\mathrm{d}\tau \int_{\mathbb{R}^d} g(s,x,\tau,z)(b(\tau,z),w(\tau,z,t,\varphi))\,\mathrm{d}z+ \nonumber \\
		\int_{\mathbb{R}^d}g(s,x,u,y)\, \mathrm{d}y \int_u^t\mathrm{d}\tau \int_{\mathbb{R}^d} g(u,y,\tau,z)(b(\tau,z),w(\tau,z,t,\varphi))\,\mathrm{d}z= \nonumber \\
		\mathds{T}_{su}^0(\mathds{T}_{ut}\varphi)(x)+
		\int_s^u\mathrm{d}\tau \int_{\mathbb{R}^d} g(s,x,\tau,z)(b(\tau,z),w(\tau,z,t,\varphi))\,\mathrm{d}z, \nonumber
	\end{align}
	where
	\begin{equation}\label{BO_eq22}
		\mathds{T}_{st}^0\varphi(x)=\int_{\mathbb{R}^d}g(s,x,t,y)\varphi(y)\,\mathrm{d}y.
	\end{equation}
	Using \eqref{BO_eq21}, one can obtain (by changing of integration order)
	\[
	w(s,x,t,\varphi)=\int_{\mathbb{R}^d}\varphi(y)\,\mathrm{d}y \int_{\mathbb{R}^d}\nabla_{\beta}g(s,\cdot,u,z)(x)g(u,z,t,y)\,\mathrm{d}z+
	\]
	\begin{align}\int_s^u\mathrm{d}\tau \int_{\mathbb{R}^d} v_0(s,x,\tau,z)(b(\tau,z),w(\tau,z,t,\varphi))\,\mathrm{d}z+ \nonumber \\
		\int_u^t\mathrm{d}\tau \int_{\mathbb{R}^d} \int_{\mathbb{R}^d} v_0(s,x,u,y)g(u,y,\tau,z)\,\mathrm{d}y(b(\tau,z),w(\tau,z,t,\varphi))\,\mathrm{d}z= \nonumber \\
		w_0(s,x,u,\mathds{T}_{ut}\varphi)+ \int_s^u\mathrm{d}\tau \int_{\mathbb{R}^d} v_0(s,x,\tau,z)(b(\tau,z),w(\tau,z,t,\varphi))\,\mathrm{d}z. \nonumber
	\end{align}
	Here we used the well-known relation (the Chapman-Kolmogorov equation)
	\[
	g(s,x,t,y)=\int_{\mathbb{R}^d}g(s,x,u,z)g(u,z,t,y)\,\mathrm{d}z,
	\]
	which is true for all $0\le s<u<t$, $x, y\in\mathbb{R}^d$.
	
	Since equation \eqref{BO_eq21} has a unique solution, we state that 
	\[
	w(s,x,t,\varphi)=w(s,x,u,\mathds{T}_{ut}\varphi)
	\]
	for all $\varphi\in C_b(\mathbb{R}^d)$, $x\in\mathbb{R}^d$, $0\le s<u<t$. Therefore $\mathds{T}_{st}\varphi(x)=\mathds{T}_{su}(\mathds{T}_{ut}\varphi)(x)$ and the evolution property is proved.
	
	The last statement of this theorem follows from the following two facts: first, using \eqref{BO_eq22} we have the equality $\wlim_{s\uparrow t}\mathds{T}_{st}^0=I$, and second, (note that $1\le q<\frac{d+\alpha}{d+1}$)
	\begin{align}
		\left|\int_s^t\mathrm{d}\tau \int_{\mathbb{R}^d} g(s,x,\tau,z)(b(\tau,z),w(\tau,z,t,\varphi))\,\mathrm{d}z\right|\le \nonumber \\
		\|b\|_p^TN_{0,T}C_T\left(\int_s^t\mathrm{d}\tau \int_{\mathbb{R}^d} \frac{(\tau-s)^{(1-\gamma/\alpha)q}}{((\tau-s)^{1/\alpha}+|x-z|)^{(d+\alpha-\gamma)q}}(t-\tau)^{-q\beta/\alpha}\,\mathrm{d}z\right)^{1/q}\le \nonumber \\
		{C}_T\left(\int_s^t  (\tau-s)^{-(q-1)d/\alpha}(t-\tau)^{-q\beta/\alpha}\,\mathrm{d}\tau\right)^{1/q}= \nonumber \\
		{C}_T(t-s)^{1-q\beta/\alpha-(q-1)d/\alpha}\to0,\mbox{ as }s\uparrow t. \nonumber 
	\end{align}
	This completes the proof of the theorem.
\end{proof}
\begin{remark}
	We cannot state that the operators $\mathds{T}_{st}$ keep the cone of non-negative functions. We have no proof that the function $\mathds{T}_{st}\varphi(x)$ can have negative values if $\varphi(x)\ge0$, $x\in\mathbb{R}^d$. But the example of $\alpha$-stable process and $b(t,x)\equiv b\in\mathbb{R}^d$ confirms this fact. 
	Exactly analogous to how it is done in \cite{BigOs} for the case $\beta=\alpha-1$, we can obtain the following equality
	\[
	G(s,x,t,y)=\frac{1}{(2\pi)^d}\int_{\mathbb{R}^d}\exp\{i(x-y-(t-s)b|\lambda|^{\beta-1},\lambda)-c(t-s)|\lambda|^\alpha\}\mathrm{d}\lambda,
	\]
	in our case. The function $G$ here has negative values as the Fourier transform of a non-positive definite function.
	
	Thus, the evolution operators family $\{\mathds{T}_{st}:0\le s<t\}$ does not define any of the Markov processes but only a pseudo-process possessing a ``Markov property''.
\end{remark}

\section{Cauchy problem}\label{BO_sec5}
In this section we fix some $T>0$ and prove that the function $u(s,x,t)=\mathds{T}_{st}\varphi(x)$, $0\le s< t\le T$, $x\in\mathbb{R}^d$ is a solution (in some sense) to the following Cauchy problem:
\begin{align}
	\frac{\partial}{\partial s}u(s,x,t)+L(s,x)u(s,\cdot,t)(x)=0,\quad 0\le s<t,\ x\in\mathbb{R}^d\label{BO_eq23} \\
	\lim_{s\uparrow t}u(s,x,t)=\varphi(x),\quad x\in\mathbb{R}^d,\label{BO_eq24}
\end{align}
for every $t\in(0,T]$, where $L(s,x)=A(s,x)+(b(s,x),\nabla_{\beta})$ and $\varphi\in C_b(\mathbb{R}^d)$.

Equality \eqref{BO_eq24} is proved in Theorem \ref{BO_th2} (see the last statement). Equality \eqref{BO_eq23} will be proved in some generalized sense. Namely, let a sequence $\{b_n:n\in\mathbb{N}\}\subset L_p([0,T]\times\mathbb{R}^d)$ be such that $L_p\mbox{-}\lim_{n\to\infty}b_n=b$ and $u_n(s,x,t)=\mathds{T}_{st}^{(n)}\varphi(x)$ satisfied the Cauchy problem \eqref{BO_eq23}, \eqref{BO_eq24} with the function $b_n$ instead of $b$. The operators $\mathds{T}_{st}^{(n)}$ are constructed using Theorems \ref{BO_th1}, \ref{BO_th2} and functions $b_n$ instead of $b$.  If the limit $u(s,x,t)=\lim_{n\to\infty}u_n(s,x,t)$, $0\le s<t\le T$, $x\in\mathbb{R}^d$ exists, we will call the function $u$ by the generalized solution to the Cauchy problem \eqref{BO_eq23}, \eqref{BO_eq24}. 
The following auxiliary statement will be useful in constructing this generalized solution.

\begin{lemma}
	\label{BO_lem5}
	Let functions $(\tilde{b}(s,x))_{s\ge0,x\in\mathbb{R}^d}$ and $(\hat{b}(s,x))_{s\ge0,x\in\mathbb{R}^d}$ satisfy the assumptions of Theorem \ref{BO_th1} with the same $p$. Then the corresponding functions $\tilde{G}$ and $\hat{G}$ satisfy the following inequality
	\begin{align}
		\label{BO_eq25}
		|\tilde{G}(s,x,t,y)-\hat{G}(s,x,t,y)|\le C_T\|\tilde{b}-\hat{b}\|_p^T(1+\|\tilde{b}\|_p^T+\|\hat{b}\|_p^T)\times\nonumber\\
		\frac{1}{((t-s)^{1/\alpha}+|x-y|)^{d+\beta-\gamma}},
	\end{align}
	for all $0\le s<t\le T$, $x,y\in\mathbb{R}^d$ and each $T>0$ with some positive constant $C_T$.
\end{lemma}
\begin{proof}
	Using \eqref{BO_eq10}, we can obtain the following
	\begin{equation}\label{BO_eq26}
		\hat{G}(s,x,t,y)-\tilde{G}(s,x,t,y)=\int_s^t\mathrm{d}\tau \int_{\mathbb{R}^d} g(s,x,\tau,z)W(\tau,z,t,y))\,\mathrm{d}z
	\end{equation}
	for all $0\le s<t$, $x,y\in\mathbb{R}^d$, where $W(s,x,t,y)=(\hat{b}(s,x),\hat{v}(s,x,t,y))- (\tilde{b}(s,x),\tilde{v}(s,x,t,y))$, in which $\hat{v}$ and $\tilde{v}$ are solutions to equations obtained from \eqref{BO_eq11} replacing the function $b$ by the functions $\hat{b}$ and $\tilde{b}$, respectively. Relation \eqref{BO_eq11} leads us to the equation
	\begin{align}
		W(s,x,t,y)=W_0(s,x,t,y)+\int_s^t\mathrm{d}\tau \int_{\mathbb{R}^d}(v_0(s,x,\tau,z),\hat{b}(s,x))W(\tau,z,t,y))\,\mathrm{d}z+ \nonumber\\
		\int_{\mathbb{R}^d}W_0(s,x,\tau,z)(\tilde{b}(\tau,z),\tilde{v}(\tau,z,t,y))\,\mathrm{d}z,\nonumber
	\end{align}
	where $W_0(s,x,t,y)=(\hat{b}(s,x)-\tilde{b}(s,x),v_0(s,x,t,y))$. Considering \eqref{BO_eq4}, we obtain the following inequality
	\begin{align}
		|W_0(s,x,t,y)\le|\hat{b}(s,x)-\tilde{b}(s,x)|N_{\beta,T}\left(\frac{1}{((t-s)^{1/\alpha}+|y-x|)^{d+\beta}}\right.+\nonumber\\
		\left.\frac{(t-s)^{1-\beta/\alpha}}{((t-s)^{1/\alpha}+|y-x|)^{d+\alpha-\gamma}}\right),\nonumber
	\end{align}	
	valid for all $0\le s<t\le T$, $x,y\in\mathbb{R}^d$	and each $T>0$. Moreover
	\begin{align}
		\left|\int_s^t\mathrm{d}\tau\int_{\mathbb{R}^d}W_0(s,x,\tau,z)(\tilde{b}(\tau,z),\tilde{v}(\tau,z,t,y))\,\mathrm{d}z\right|\le\nonumber \\
		N_{\beta,T} |\hat{b}(s,x)-\tilde{b}(s,x)| C_T \int_s^t\mathrm{d}\tau \int_{\mathbb{R}^d}|\tilde{b}(\tau,z)|\times \nonumber\\
		\left(\frac{1}{((\tau-s)^{1/\alpha}+|z-x|)^{d+\beta}}+
		\frac{(\tau-s)^{1-\beta/\alpha}}{((\tau-s)^{1/\alpha}+|z-x|)^{d+\alpha-\gamma}}\right)\times\nonumber\\
		\left(\frac{1}{((t-\tau)^{1/\alpha}+|y-z|)^{d+\beta}}+
		\frac{(t-\tau)^{1-\beta/\alpha}}{((t-\tau)^{1/\alpha}+|y-z|)^{d+\alpha-\gamma}}\right)\,\mathrm{d}z=\nonumber\\
		C_T  |\hat{b}(s,x)-\tilde{b}(s,x)| \sum_{k=1}^{4}I_k,\nonumber
	\end{align}
	where
	\[
	I_1=\int_s^t\mathrm{d}\tau \int_{\mathbb{R}^d}|\tilde{b}(\tau,z)|
	\frac{1}{((\tau-s)^{1/\alpha}+|z-x|)^{d+\beta}}
	\frac{1}{((t-\tau)^{1/\alpha}+|y-z|)^{d+\beta}}\,\mathrm{d}z,
	\]	
	\begin{align}
		I_2=\int_s^t\mathrm{d}\tau \int_{\mathbb{R}^d}|\tilde{b}(\tau,z)|
		\frac{1}{((\tau-s)^{1/\alpha}+|z-x|)^{d+\beta}}\times\nonumber\\
		\frac{(t-\tau)^{1-\beta/\alpha}}{((t-\tau)^{1/\alpha}+|y-z|)^{d+\alpha-\gamma}}\,\mathrm{d}z,\nonumber
	\end{align}	
	\begin{align}
		I_3=\int_s^t\mathrm{d}\tau \int_{\mathbb{R}^d}|\tilde{b}(\tau,z)|
		\frac{(\tau-s)^{1-\beta/\alpha}}{((\tau-s)^{1/\alpha}+|z-x|)^{d+\alpha-\gamma}}\times\nonumber\\
		\frac{1}{((t-\tau)^{1/\alpha}+|y-z|)^{d+\beta}}\,\mathrm{d}z,\nonumber
	\end{align}	
	\begin{align}
		I_4=\int_s^t\mathrm{d}\tau \int_{\mathbb{R}^d}|\tilde{b}(\tau,z)|
		\frac{(\tau-s)^{1-\beta/\alpha}}{((\tau-s)^{1/\alpha}+|z-x|)^{d+\alpha-\gamma}}\times\nonumber\\
		\frac{(t-\tau)^{1-\beta/\alpha}}{((t-\tau)^{1/\alpha}+|y-z|)^{d+\alpha-\gamma}}\,\mathrm{d}z.\nonumber
	\end{align}	
	Since $\tilde{b}\in L_p([0,T]\times\mathbb{R}^d)$ for each $T>0$, using inequality \eqref{BO_eq2} we obtain the following estimates
	\[
	I_1\le\|\tilde{b}\|_p^T (2CB(1,\theta q))^{1/q}\frac{(t-s)^{\theta}}{((t-s)^{1/\alpha}+|y-x|)^{d+\beta}},
	\]
	\begin{align}
		I_2\le\|\tilde{b}\|_p^T \left(2C\max\left(B\left(1,\theta q+q\frac{\gamma}{\alpha}\right), B\left(1+\left(1-\frac{\beta}{\alpha}\right)q,\theta q\right) \right)\right)^{1/q}\times\nonumber\\ 
		\left( \frac{(t-s)^{\theta+\gamma/\alpha}}{((t-s)^{1/\alpha}+|y-x|)^{d+\beta}}+ \frac{(t-s)^{\theta+1-\beta/\alpha}}{((t-s)^{1/\alpha}+|y-x|)^{d+\alpha-\gamma}}\right),\nonumber
	\end{align}
	\begin{align}
		I_3\le\|\tilde{b}\|_p^T \left(2C\max\left(B\left(1,\theta q+q\frac{\gamma}{\alpha}\right), B\left(1+\left(1-\frac{\beta}{\alpha}\right)q,\theta q\right) \right)\right)^{1/q}\times\nonumber\\ 
		\left( \frac{(t-s)^{\theta+\gamma/\alpha}}{((t-s)^{1/\alpha}+|y-x|)^{d+\beta}}+ \frac{(t-s)^{\theta+1-\beta/\alpha}}{((t-s)^{1/\alpha}+|y-x|)^{d+\alpha-\gamma}}\right),\nonumber
	\end{align}
	\[
	I_4\le\|\tilde{b}\|_p^T \left(2CB\left(1+\left(1-\frac{\beta}{\alpha}\right)q,\theta q+q\frac{\gamma}{\alpha}\right)\right)^{1/q}
	\frac{(t-s)^{\theta+1-\beta/\alpha+\gamma/\alpha}}{((t-s)^{1/\alpha}+|y-x|)^{d+\alpha-\gamma}},\nonumber
	\]
	where, as it was above, $\theta=1-((d+\alpha)/p+\beta)/\alpha$, $q=p/(p-1)$ and $C$ is maximum of  positive constants derived from inequality \eqref{BO_eq2} in considered four cases. Therefore
	\begin{align}
		\left|\int_s^t\mathrm{d}\tau\int_{\mathbb{R}^d}W_0(s,x,\tau,z)(\tilde{b}(\tau,z),\tilde{v}(\tau,z,t,y))\,\mathrm{d}z\right|\le C_T\|\tilde{b}\|_p^T |\hat{b}(s,x)-\tilde{b}(s,x)|\times \nonumber \\
		\left(\frac{1}{((t-s)^{1/\alpha}+|y-x|)^{d+\beta}}+
		\frac{(t-s)^{1-\beta/\alpha}}{((t-s)^{1/\alpha}+|y-x|)^{d+\alpha-\gamma}}\right).\nonumber	
	\end{align}
	
	Let us denote
	\[
	W_0^*(s,x,t,y)=W_0(s,x,t,y)+\int_s^t\mathrm{d}\tau\int_{\mathbb{R}^d}W_0(s,x,\tau,z)(\tilde{b}(\tau,z),\tilde{v}(\tau,z,t,y))\,\mathrm{d}z.
	\]
	Then we can write down the following
	\begin{equation}
		\label{BO_eq27}
		W(s,x,t,y)=W_0^*(s,x,t,y)+\int_s^t\mathrm{d}\tau \int_{\mathbb{R}^d}(v_0(s,x,\tau,z),\hat{b}(s,x))W(\tau,z,t,y))\,\mathrm{d}z.
	\end{equation}
	This equation can be solved by the method of successive approximations, i. e., a solution of it can be found in the form $W(s,x,t,y)=\sum_{k=0}^\infty W_k^*(s,x,t,y)$. The terms of this series satisfy the following relation
	\[
	W_k^*(s,x,t,y)=\int_s^t\mathrm{d}\tau \int_{\mathbb{R}^d}(v_0(s,x,\tau,z),\hat{b}(s,x))W_{k-1}^*(\tau,z,t,y))\,\mathrm{d}z \quad k=1,2,\dots.
	\]
	To justification this, note that
	\begin{align}
		|W_0^*(s,x,t,y)|\le C_T(1+\|\tilde{b}\|_p^T )|\hat{b}(s,x)-\tilde{b}(s,x)|\times\nonumber\\
		\left(\frac{1}{((t-s)^{1/\alpha}+|y-x|)^{d+\beta}}+
		\frac{(t-s)^{1-\beta/\alpha}}{((t-s)^{1/\alpha}+|y-x|)^{d+\alpha-\gamma}}\right).\nonumber
	\end{align}
	Moreover, by the method mathematical induction one can prove the following estimate
	\begin{align}
		|W_k^*(s,x,t,y)|\le R_k|\hat{b}(s,x)|\|\hat{b}-\tilde{b}\|_p^T\times\nonumber\\ \left(\frac{(t-s)^{k\theta}}{((t-s)^{1/\alpha}+|y-x|)^{d+\beta}}+
		\frac{(t-s)^{k\theta+1-\beta/\alpha}}{((t-s)^{1/\alpha}+|y-x|)^{d+\alpha-\gamma}}\right).\nonumber
	\end{align}
	valid for all $k\ge1$, where $R_1=C_T(1+\|\tilde{b}\|_p^T)$ and $C_T$ is some positive constant depended on $T$, and, for $k\ge2$
	\begin{align}
		R_k=R_{k-1}C_T(1+T^{\gamma/\alpha})(1+\|\tilde{b}\|_p^T )\|\hat{b}\|_p^T \times\nonumber\\
		(8C\max(B(1,k\theta q),B(\theta q, (k-1)\theta q+1)))^{1/q}.\nonumber
	\end{align}
	Therefore the series $\sum_{k=0}^\infty W_k^*(s,x,t,y)$ converges uniformly with respect to $x,y\in\mathbb{R}^d$ and locally uniformly with respect to $0\le s<t$. So, its sum $W$ is a solution to equation \eqref{BO_eq27}. In addition we obtain the following estimate: 
	\begin{align}
		|W(s,x,t,y)|\le {C}_T\left((1+\|\tilde{b}\|_p^T)|\hat{b}(s,x)-\tilde{b}(s,x)|+|\hat{b}(s,x)|\|\hat{b}-\tilde{b}\|_p^T\right)\times\nonumber\\
		\left(\frac{1}{((t-s)^{1/\alpha}+|y-x|)^{d+\beta}}+
		\frac{(t-s)^{1-\beta/\alpha}}{((t-s)^{1/\alpha}+|y-x|)^{d+\alpha-\gamma}}\right),\label{BO_eq28}
	\end{align}
	valid for all $0\le s<t\le T$, $x,y\in\mathbb{R}^d$	and each $T>0$, where ${C}_T$ is some positive constant depended on $T$.
	
	Using \eqref{BO_eq26},\eqref{BO_eq28} and \eqref{BO_eq2} with the H\"{o}lder inequality, we obtain \eqref{BO_eq25}.
\end{proof}

\begin{lemma}
	\label{BO_lem6}
	Let the function $w$ be defined in Lemma \ref{BO_lem4}. For every $0\le s<t\le T$, $x, y\in\mathbb{R}^d$, the following inequality
	\[
	|w(s,x,t,\varphi)-w(s,y,t,\varphi)| \le C_T|x-y|(t-s)^{-(\beta+1)/\alpha},
	\]
	holds with some positive constant $C_T$ depended on $T$.
\end{lemma}
\begin{proof}
	Using \eqref{BO_eq21} we can write down the relation ($0\le s<t$, $x, y\in\mathbb{R}^d$)
	\begin{align}\label{BO_eq29}
		w(s,x,t,\varphi)-w(s,y,t,\varphi)=w_0(s,x,t,\varphi)-w_0(s,y,t,\varphi)+\nonumber\\
		\int_s^t\mathrm{d}\tau \int_{\mathbb{R}^d} (v_0(s,x,\tau,z)-v_0(s,y,\tau,z))(b(\tau,z),w(\tau,z,t,\varphi))\,\mathrm{d}z.
	\end{align}
	Let us remark that (see \eqref{BO_eq5})
	\[
	v_0(s,x,t,y)=\nabla_\beta g_0(s,\cdot,t,y)(x)+\int_s^t\mathrm{d}\tau \int_{\mathbb{R}^d} \nabla_\beta g_0(s,\cdot,\tau,z)(x)\Phi(\tau,z,t,y)\,\mathrm{d}z.
	\]
	Moreover, for all $0\le s<t$, $x,y,z\in\mathbb{R}^d$,
	\begin{align}
		\nabla_\beta g_0(s,\cdot,t,z)(x)-\nabla_\beta g_0(s,\cdot,t,z)(y) = \nonumber\\
		\frac{i}{(2\pi)^d}\int_{\mathbb{R}^d}\left(e^{i(x-z,\lambda)}-e^{i(y-z,\lambda)}\right) \lambda|\lambda|^{\beta-1}\exp\left\{-a(t,z,\lambda)(t-s)\right\}\,\mathrm{d}\lambda = \nonumber\\
		-\frac{1}{(2\pi)^d}\int_{\mathbb{R}^d}e^{i(\eta-z,\lambda)}(\lambda,x-y) \lambda|\lambda|^{\beta-1}\exp\left\{-a(t,z,\lambda)(t-s)\right\}\,\mathrm{d}\lambda = \nonumber\\
		|x-y|D_{\beta+1}g_0(s,\cdot,t,z)(\eta),\nonumber
	\end{align}
	where $\eta=\theta x+(1-\theta)y$ with some $\theta\in(0,1)$ and the operator $D_{\beta+1}$ is defined by the symbol $\left(\left(\lambda,\frac{x-y}{|x-y|}\right)\lambda|\lambda|^{\beta-1}\right)_{\lambda\in\mathbb{R}^d}$. This symbol satisfies the assumptions of \cite[Lemma 4.2]{Kochubey}, which leads us to the estimate
	\[
	|D_{\beta+1}g_0(s,\cdot,t,z)(\eta)|\le\frac{C}{((t-s)^{1/\alpha}+|z-\eta|)^{d+\beta+1}}.
	\]
	Thus, using \eqref{BO_eq2}, we obtain the inequality
	\begin{align}
		|v_0(s,x,t,z)-v_0(s,y,t,z)|\le \frac{C|x-y|}{((t-s)^{1/\alpha}+|z-\eta|)^{d+\beta+1}}+	\nonumber\\
		\int_s^t\mathrm{d}\tau \int_{\mathbb{R}^d} \frac{C|x-y|}{((\tau-s)^{1/\alpha}+|u-\eta|)^{d+\beta+1}}\frac{C}{((t-\tau)^{1/\alpha}+|z-u|)^{d+\alpha-\gamma}}\,\mathrm{d}u\le	\nonumber\\
		|x-y| C_T\left(\frac{1}{((t-s)^{1/\alpha}+|z-\eta|)^{d+\beta+1}}+ \frac{(t-s)^{1-(\beta+1)/\alpha}}{((t-s)^{1/\alpha}+|z-\eta|)^{d+\alpha-\gamma}}\right),	\nonumber
	\end{align}
	valid for all $0\le s<t\le T$, $x,y,z\in\mathbb{R}^d$. And, as consequence, the following inequalities are true for all $0\le s<t\le T$, $x,y\in\mathbb{R}^d$ and $\varphi\in C_b(\mathbb{R}^d)$:
	\[
	|v_0(s,x,t,\varphi)-v_0(s,y,t,\varphi)|\le C_T|x-y|\|\varphi\|(t-s)^{-(\beta+1)/\alpha};
	\]
	\begin{align}
		\left|
		\int_s^t\mathrm{d}\tau \int_{\mathbb{R}^d} (v_0(s,x,\tau,z)-v_0(s,y,\tau,z))(b(\tau,z),w(\tau,z,t,\varphi))\,\mathrm{d}z
		\right|\le
		\nonumber\\
		C_T|x-y|\|b\|_p^T\left(\int_s^t(t-\tau)^{-\beta q/\alpha}\mathrm{d}\tau \int_{\mathbb{R}^d}\left(\frac{1}{((t-s)^{1/\alpha}+|z-\eta|)^{d+\beta+1}}+\right.\right.
		\nonumber\\
		\left.\left.\frac{(t-s)^{1-(\beta+1)/\alpha}}{((t-s)^{1/\alpha}+|z-\eta|)^{d+\alpha-\gamma}}\right)^q\,\mathrm{d}z\right)^{1/q}\le\nonumber\\ 
		{C}_T|x-y|(t-s)^{1-(1+d/\alpha)/p-(2\beta+1)/\alpha},
		\nonumber
	\end{align}
	where $q=p/(p-1)$.
	
	Therefore, using \eqref{BO_eq29} and the fact that $1-(1+d/\alpha)/p-\beta/\alpha>(1-\beta)/\alpha$, we obtain the lemma statement.
\end{proof}

In the next theorem, we construct a generalized solution to the Cauchy problem formulated at the beginning of the section.

\begin{thm}
	\label{BO_th3}
	Let the assumptions of Theorem \ref{BO_th1} hold and the function $G$ is constructed there. Then the function
	\[
	u(s,x,t)=\int_{\mathbb{R}^{d}}G(s,x,t,y)\varphi(y)\,\mathrm{d}y,\quad 0\le s<t\le T,\ x\in\mathbb{R}^d
	\]
	is a generalized solution to the Cauchy problem \eqref{BO_eq23}, \eqref{BO_eq24} for each  $\varphi\in C_b(\mathbb{R}^d)$.
\end{thm}
\begin{proof}
	Let us consider a sequence $\left\{b_n:n\in\mathbb{N}\right\}$ of $\mathbb{R}^d$-valued functions, which are infinity continuous differentiable, have compact supports and belong to $L_p([0,T]\times\mathbb{R}^d)$, where $p$ is defined in Theorem \ref{BO_th1}. Assume that $L_p\mbox{-}\lim_{n\to\infty}b_n=b$, where $b$ is the function from Theorem \ref{BO_th1}. 
	
	We denote by $v_n$, $w_n$ and $G_n$ the objects that are defined as  $v$, $w$ and $G$, respectively, using $b_n$ instead of $b$.

	Lemma \ref{BO_lem5} allows us to state that the sequence of corresponding functions $G_n$ constructed in Theorem \ref{BO_th1} using the functions $b_n$ instead of $b$ converges to the function $G$ uniformly with respect to $y\in\mathbb{R}^d$ for each fixed $0\le s<t\le T$ and $x\in\mathbb{R}^d$.

	Let us consider the function 
	\[
	f_n(s,x,t)=\int_{\mathbb{R}^{d}}(b_n(s,x),\nabla_{ \beta}G_n(s,\cdot,t,y)(x))\varphi(y)\,\mathrm{d}y, \quad 0\le s<t\le T,\ x\in\mathbb{R}^d.
	\] 
	Remind that $\nabla_{ \beta}G_n(s,\cdot,t,y)(x)=v_n(s,x,t,y)$. Moreover, 
	\begin{align}
		|f_n(s,x,t)-f_n(s,y,t)|\le |b_n(s,x)||w_n(s,x,t,\varphi)-w_n(s,y,t,\varphi)|+\nonumber\\ |b_n(s,x)-b_n(s,y)||w_n(s,y,t,\varphi)|\nonumber
	\end{align}
	Since the function $b_n$ is Lipschitz continuous in $x$, uniformly to $s$ and has a compact support, then, as follows from Lemma \ref{BO_lem6}, the function $f_n(s,x,t)$
	is Lipschitz continuous in $x$, uniformly to $s\in[0,t)$ for every fixed $t>0$.
	
	Therefore (see \cite[Th. 4.1.]{Kochubey}) the function $u_n(s,x,t)=\int_{\mathbb{R}^{d}}G_n(s,x,t,y)\varphi(y)\,\mathrm{d}y$ is a solution to the Cauchy problem \eqref{BO_eq23}, \eqref{BO_eq24} for every $t\in(0,T]$.

\end{proof}


\end{document}